\pdfoutput=1
\documentclass[10pt,letterpaper,leqno,twoside]{article}
\usepackage[colorlinks=true, pdfstartview=FitV, linkcolor=blue, citecolor=blue, urlcolor=blue]{hyperref}
\usepackage[latin1]{inputenc}
\usepackage[OT2,OT1]{fontenc}
\DeclareFontFamily{OT1}{rsfs}{}
\DeclareFontShape{OT1}{rsfs}{n}{it}{<-> rsfs10}{}
\DeclareMathAlphabet{\mathscr}{OT1}{rsfs}{n}{it}
\usepackage[english]{babel}
\usepackage{amsmath}
\usepackage{amsfonts}
\usepackage{amssymb}
\usepackage{amsthm}
\usepackage{graphicx}
\usepackage[protrusion=true,expansion=true]{microtype}
\usepackage[left=2.5cm,right=2.5cm,top=3cm,bottom=3cm]{geometry}
\author{Dimitri Dias}
\title{A short proof of a known result about the density of a certain set in $[0,1]^n$}
\date{}

\newtheorem{theorem}{Theorem}
\newtheorem{lemma}{Lemma}
\newtheorem{corollary}{Corollary}

\newcounter{numrem}
\newenvironment{remark}[1][Remark \arabic{numrem} :]{\refstepcounter{numrem} \begin{trivlist}
\item[\hskip \labelsep {\bfseries #1}]}{\end{trivlist}}

\begin{document}
\maketitle

\abstract{In Theorem 1 of \cite{Cobel}, Cobeli and Zaharescu give a result about the distribution of the {${\bf F}_p$}-points on an affine curve. An easy corollary to their theorem is  that the set
\begin{equation*}
\bigcup_p \left\lbrace \left( \frac{x_1}{p}, \dots ,\frac{x_n}{p} \right), 1 \leq x_i < p \text{ and } \prod_{1 \leq i \leq n} x_i \equiv 1 \bmod{p} \right\rbrace
\end{equation*}
is dense in $\left[ 0,1 \right]^n$. In \cite{Foo}, Foo gives a elementary proof of that fact in dimension $2$. Following Foo's ideas, we give a similar proof in dimension greater than or equal to $3$.}

\section{Introduction}

In Theorem 1 of \cite{Cobel}, Cobeli and Zaharescu give a result about the distribution of the {${\bf F}_p$}-points on an affine curve. In dimension $n$, for any curve $\mathcal{C}$ over $\mathbb{F}_p$ not contained in any hyperplane, for any nice domain $\Omega$ in the torus $\mathbb{T}^n$ and for any prime $p$, let $\mu$ be the normalized Haar measure on $\mathbb{T}^n$ and $\mu_{n,p,\mathcal{C}}=\frac{1}{\left| \mathcal{C}(\mathbb{F}_p) \right|} \sum_{x \in \mathcal{C}(\mathbb{F}_p)} \delta_{t(x)}$, with $\delta_{t(x)}$ a unit point delta mass at $t(x)$, where $t$ is a natural injection from $\mathbb{F}_p^n$ to $\mathbb{T}^n$. Cobeli and Zaharescu quantify how fast $\mu_{n,p,\mathcal{C}}(\Omega)$ approaches $\mu(\Omega)$ and their result easily imply that the set
\begin{equation*}
\mathcal{A}_n = \bigcup_p \left\lbrace \left( \frac{x_1}{p}, \dots ,\frac{x_n}{p} \right), 1 \leq x_i < p \text{ and } \prod_{1 \leq i \leq n} x_i \equiv 1 \bmod{p} \right\rbrace
\end{equation*}
is dense in $\left[ 0,1 \right]^n$.
\\
\\
Their proof mainly uses exponential sums and, as remarked by Foo in \cite{Foo}, one can give an elementary proof of the previous fact in dimension $2$. Following his ideas, we prove the result in dimension greater than or equal to $3$.

\begin{theorem} \label{Thm1}
Let $n \geq 3$. The set $\mathcal{A}_n$ is dense in $\left[ 0,1 \right]^n$.
\end{theorem}

\subsection*{Acknowledgements}
I would like to thank my supervisor, Professor Andrew Granville, for discussions that greatly helped me with my work. I would also like to thank my friends Marc Munsch, Oleksiy Klurman, Crystel Bujold and Marzieh Mehdizadeh for their helpful comments.

\newpage

\section{Proof of Theorem \ref{Thm1}}

We will need the following lemmas:

\begin{lemma} \label{lem1}
Let $x \in \left[ 0,1 \right]$ and $0 < \varepsilon \leq 1$. Let $N$ such that $N > \frac{1}{\varepsilon}$ and $c_3 \frac{(\log N)^{c_2}}{N} < \frac{\varepsilon}{2}$ (where $c_2$ and $c_3$ are absolute constants defined in the proof). Then, for every $b \geq N$, there exists $1 \leq a < b$ with $(a,b)=1$, $a > \frac{\varepsilon}{2} b$ and $\left| x- \frac{a}{b} \right| < \varepsilon$.
\end{lemma}

\begin{proof}
Let $b \geq N$ and consider the set
\begin{equation*}
\mathcal{D}_b = \left\lbrace \frac{a}{b} \text{ with } 1 \leq a < b \text{ and } (a,b)=1 \right\rbrace \, .
\end{equation*}
Let $g(b)$ be the least integer such that every set of $g(b)$ consecutive integers contains at least one integer relatively prime to $b$. As remarked by Erd\H{o}s in \cite{Erd}, a standard application of Brun's method gives that
\begin{equation*}
g(b) \leq c_1 \omega(b)^{c_2}
\end{equation*}
where $\omega(b)$ denotes the number of prime factors of $b$ and $c_1$, $c_2$ are absolute constants. Therefore, the distance between two consecutive elements of $\mathcal{D}_b$ is bounded above by
\begin{equation*}
\frac{g(b)+1}{b} \leq \frac{c_1 \omega(b)^{c_2} +1}{b} \leq c_3 \frac{(\log b)^{c_2}}{b} < \frac{\varepsilon}{2} \, .
\end{equation*}
The minimal value of this set is $\frac{1}{b} < \varepsilon$ and the maximal value is $1-\frac{1}{b} > 1 - \varepsilon$. Therefore, for any $x \in \left[ 0,1 \right]$, there exists $\frac{a}{b} \in \mathcal{D}_b$ with $\frac{a}{b} > \frac{\varepsilon}{2}$ and $\left| x- \frac{a}{b} \right| < \varepsilon$.
\end{proof}

\begin{lemma} \label{lemm}
The set
\begin{equation*}
\mathcal{B}_n = \left\lbrace \left( \frac{a_0}{a_1}, \dots ,\frac{a_{n-1}}{a_n} \right), 1 \leq a_i < a_{i+1}, (a_i, a_{i+1})=1 \text{ and } (a_1, a_2 \dots a_n)=1 \right\rbrace
\end{equation*}
is dense in $\left[ 0,1 \right]^n$.
\end{lemma}

\begin{proof}
Let $(x_1, \dots, x_n) \in \left[ 0,1 \right]^n$ and $\varepsilon > 0$. We can also assume $\varepsilon \leq 1$. Let $M$ such that ${c_3 \frac{\log^{c_2} M}{M} < \left( \frac{\varepsilon}{2} \right)^{n-1} \frac{1}{n^{c_2}}}$ and $M > \frac{4}{\varepsilon}$ (note in particular that $M$ can be supposed larger than the $N$ in Lemma \ref{lem1}).
\\
\\
Choose $a_{n-1}$ a prime number such that $a_{n-1} \geq \left( \frac{2}{\varepsilon} \right)^{n-2}M$. Let
\begin{equation*}
\mathcal{H}_{a_{n-1}} = \left\lbrace \frac{a_{n-1}}{m} \text{ with } a_{n-1} < m \text{ and } (a_{n-1},m)=1 \right\rbrace \, .
\end{equation*}
Since $a_{n-1}$ is prime, if $(a_{n-1},m)=1$, then either $(a_{n-1},m+1)=1$ or $(a_{n-1},m+2)=1$. Therefore, the distance between two consecutive elements in the set $\mathcal{H}_{a_{n-1}}$ is bounded above by
\begin{equation*}
\frac{2}{a_{n-1}} < \frac{\varepsilon}{2} \, .
\end{equation*}
The set $\mathcal{H}_{a_{n-1}}$ contains arbitrarily small elements and has maximum $1-\frac{1}{a_{n-1}} > 1 - \varepsilon$. Therefore, there exists $a_n$ such that $(a_{n-1},a_n)=1$, $\frac{a_{n-1}}{a_n} > \frac{\varepsilon}{2}$ and $\left| x_n - \frac{a_{n-1}}{a_n} \right| < \varepsilon$ (we will need $a_n$ not too large in terms of $a_{n-1}$ in Equation \ref{eq1}).
\\
\\
Using Lemma \ref{lem1} $(n-3)$ times, we find $a_{n-2}, \dots , a_2$ with $(a_i,a_{i+1})=1$, $a_i > \frac{\varepsilon}{2} a_{i+1}$ and $\left| x_i - \frac{a_{i-1}}{a_i} \right| < \varepsilon$ (the choice of $a_{n-1}$ large enough in terms of $M$ allows us to apply Lemma \ref{lem1} at each step).
\\
\\
To find $a_1$, we need a slightly modified version of Lemma \ref{lem1}. Let
\begin{equation*}
\mathcal{D}_{a_2} = \left\lbrace \frac{m}{a_2} \text{ with } 1 \leq m < a_2 \text{ and } (m,a_2 \dots a_n)=1 \right\rbrace \, .
\end{equation*}
The difference between two consecutive elements in this set is bounded by
\begin{equation} \label{eq1}
\frac{g(a_2 \dots a_n)+1}{a_2} \leq c_3 \frac{\log^{c_2} (a_2 \dots a_n)}{a_2} \leq c_3 \left( \frac{2}{\varepsilon} \right)^{n-3} \frac{\log^{c_2} \left( \left( \frac{2}{\varepsilon} \right) a_{n-1}^{n-1} \right)}{a_{n-1}} < \frac{\varepsilon}{2}
\end{equation}
from our choice of $M$. The minimal value of this set is $\frac{1}{a_2} < \varepsilon$, and if we derestrict $m$ and let it go to infinity, we cover all of $\left[0,+\infty \right)$ with intervals of length at most $\frac{\varepsilon}{2}$. Therefore, one can always find $\frac{a_1}{a_2} \in \mathcal{D}_{a_2}$ such that $\left| x_2 - \frac{a_1}{a_2} \right| < \varepsilon$ and $\frac{a_1}{a_2}> \frac{\varepsilon}{2}$.
\\
\\
To find $a_0$, one can simply apply Lemma \ref{lem1} again.
\end{proof}

For the sake of completeness, we re-prove the following lemma of \cite{Foo}, which suffices to prove Theorem \ref{Thm1}.
\begin{lemma} \label{lemm2}
The set $\mathcal{A}_n$ is dense in the set $\mathcal{B}_n$.
\end{lemma}

\begin{proof}
Let $\left( \frac{a_0}{a_1}, \dots ,\frac{a_{n-1}}{a_n} \right) \in \mathcal{B}_n$. Consider the sequence (which exists, as a consequence of Dirichlet's theorem)
\begin{equation*}
\left( \frac{a_0 p + a_n}{pa_1}, \dots ,\frac{a_{n-1}(p+1)}{pa_n} \right)_p \text{ with } p \equiv -a_0^{-1} a_n \bmod{a_1} \text{ and } p \equiv -1 \bmod{a_2 \dots a_n} \, .
\end{equation*}
Then, this sequence is in $\mathcal{A}$ and converges to $\left( \frac{a_0}{a_1}, \dots ,\frac{a_{n-1}}{a_n} \right)$.
\end{proof}

\begin{proof}[Proof of Theorem \ref{Thm1}]
It's a straightforward consequence of Lemma \ref{lemm} and Lemma \ref{lemm2}.
\end{proof}

\begin{remark}
Using Theorem \ref{Thm1}, one can easily prove the following slightly more general result:
\begin{corollary}
Let $f$ be a monic polynomial of degree $d$. Then, the set
\begin{equation*}
\bigcup_p \left\lbrace \left( \frac{f(x_1)}{p^d}, \dots ,\frac{f(x_n)}{p^d} \right), 1 \leq x_i < p \text{ and } \prod_{1 \leq i \leq n} x_i \equiv 1 \bmod{p} \right\rbrace
\end{equation*}
is dense in $\left[ 0,1 \right]^n$.
\end{corollary}

\begin{proof}
Let $\Vert f \Vert$ be the absolute value of the largest coefficient of $f$. Let $(\alpha_1, \dots, \alpha_n) \in \left[ 0,1 \right]^n$ and $\varepsilon > 0$. We can assume that $\varepsilon \leq 1$. Using Theorem \ref{Thm1}, there exist $p > \frac{2d \Vert f \Vert}{\varepsilon}$ and $1 \leq x_i < p$ such that
\begin{equation*}
\left| \frac{x_i}{p} - \alpha_i^{\frac{1}{d}} \right| < \frac{\varepsilon}{2^{d+1}} \quad \forall 1 \leq i \leq n \text{ and } \prod_{1 \leq i \leq n} x_i \equiv 1 \bmod{p} \, .
\end{equation*}
Therefore,
\begin{equation*}
\left| \frac{x_i^d}{p^d} - \alpha_i \right| < \sum_{k=0}^{d-1} \binom{d}{k} \left( \frac{\varepsilon}{2^{d+1}} \right)^{d-k} < \frac{\varepsilon}{2} \quad \forall 1 \leq i \leq n
\end{equation*}
and
\begin{equation*}
\left| \frac{f(x_i)}{p^d} - \frac{x_i^d}{p^d} \right| \leq \frac{\Vert f \Vert}{p^d} \sum_{k=0}^{d-1} \left| x_i \right|^k \leq \frac{d \Vert f \Vert}{p} < \frac{\varepsilon}{2}  \quad \forall 1 \leq i \leq n \, .
\end{equation*}
This suffices to prove the result.
\end{proof}
\end{remark}

\begin{remark}
The theorem of \cite{Cobel} implies that a statement similar to Theorem \ref{Thm1} is true for a whole family of curves. It would be interesting to see if the previous elementary proof can be extended to curves other that $\prod_{1 \leq i \leq n} x_i \equiv 1 \bmod{p}$.
\end{remark}
\bigskip

\bibliographystyle{alpha}
\nocite{*}
\bibliography{densityresult}

\begin{thebibliography}{Foo07}

\bibitem[CZ01]{Cobel}
Cristian Cobeli and Alexandru Zaharescu.
\newblock On the distribution of the {${\bf F}_p$}-points on an affine curve in
  {$r$} dimensions.
\newblock {\em Acta Arith.}, 99(4):321--329, 2001.

\bibitem[Erd62]{Erd}
P.~Erd{\H{o}}s.
\newblock On the integers relatively prime to {$n$} and on a number-theoretic
  function considered by {J}acobsthal.
\newblock {\em Math. Scand.}, 10:163--170, 1962.

\bibitem[Foo07]{Foo}
Timothy Foo.
\newblock A short proof of a known density result.
\newblock {\em Integers}, 7:A7, 3, 2007.

\end{thebibliography}

\textsc{\footnotesize D\'{e}partement de math\'{e}matiques et statistiques,
Universit\'{e} de Montr\'{e}al,
CP 6128 succ. Centre-Ville,
Montr\'{e}al QC H3C 3J7, Canada
}

\textit{\small Email address: }\texttt{\small dimitrid@dms.umontreal.ca}

\end{document}